\documentclass[12pt]{amsart}
\usepackage{amsmath}
\usepackage{amsfonts} 
\usepackage{amssymb}
\usepackage[all]{xy}
\usepackage{bbm}
\usepackage{epsfig}
\usepackage{url}
\usepackage{color}
\usepackage{epic}
\usepackage{eepic}
\usepackage{graphicx}

\newtheorem{theorem}{Theorem}

\newtheorem{proposition}[theorem]{Proposition}

\newtheorem{lemma}[theorem]{Lemma}
\theoremstyle{definition}

\newcommand{\level}{{\ell}}

\renewcommand{\P}{{\mathbb P}}

\renewcommand{\int}{\operatorname{int}}

\newcommand{\Ha}{{H^{(1)}}}
\newcommand{\Ka}{{K^{(1)}}}
\newcommand{\Sa}{{S^{(1)}}}
\newcommand{\rhoa}{{\rho^{(1)}}}
\newcommand{\ba}{{b^{(1)}}}
\newcommand{\sa}{{s^{(1)}}}

\DeclareMathOperator{\Sing}{Sing}
\DeclareMathOperator{\mult}{mult}
\DeclareMathOperator{\Pic}{Pic}

\begin{document}

\title{An inequality for adjoint rational surfaces}

\author{Christian Haase}
\address{Christian Haase \\ 
  Institut f\"ur Mathematik \\
  Goethe-Universit\"at \\
  60325 Frankfurt/M\\
  Germany}
\email{haase@math.uni-frankfurt.de}
\urladdr{http://www.math.uni-frankfurt.de/\~{}haase}
\thanks{Research of Haase supported by DFG Heisenberg fellowship HA
  4383/4}

\author{Josef Schicho}
\address{Josef Schicho \\
  RICAM \\
  Austrian Academy of Sciences \\
  Altenberger Stra\ss e 69 \\
  4040 Linz \\
  Austria}
\email{Josef.Schicho@oeaw.ac.at}
\urladdr{http://www.ricam.oeaw.ac.at/research/symcomp}
\thanks{Research of Schicho supported by the FWF: F22766-N18.}

\begin{abstract}
  We generalize an inequality for convex lattice polygons -- aka toric
  surfaces -- to general rational surfaces. 
\end{abstract}

\maketitle

Our collaboration started when the second author proved an inequality
for algebraic surfaces which, when translated via the toric dictionary
into discrete geometry, yields an old inequality by Scott~\cite{Scott}
for lattice polygons.

In a previous article~\cite{HS2i+7}, we were then able to refine this
estimate on the discrete side. Here, we generalize the refinement to
(non-toric) algebraic surfaces. We use the ideas of one of the
discrete proofs.

\section{Introduction}\noindent

Let $S$ be a smooth complex algebraic surface, let $H$ be a big
and nef divisor on $S$, and let $K$ denote the canonical class of $S$.
Roughly speaking, the adjoint surface $\Sa$ of $(S,H)$ is (the minimal
resolution of) the image of $S$ in $|H+K|^*$, and the level of $(S,H)$
is the number of iterations of this adjunction process until
$H^{(\level)}$ on $S^{(\level)}$ is no longer big. (See
Section~\ref{sec:adjoined-pair} for precise definitions.)
If $S$ is rational, we prove the inequality
$$ 2\level b \le d + 9 \level^2\,, $$
where $d=H^2$ is the degree of $S$ in $|H|^*$
and $b=-HK$ is the anti-canonical degree of $H$.
Note that the inequality only makes sense for surfaces with negative
Kodaira dimension. We do not know about its validity in case of irrational
ruled surfaces (examples show that a much stronger inequality should
hold here).

For toric surfaces, the inequality was proved in~\cite{HS2i+7}, using
the toric dictionary in the following table.

\newlength{\sskip}
\setlength{\sskip}{3mm}
\begin{center} \footnotesize
  \begin{tabular}{c | c}
    lattice polygon $P$ & toric surface $X_P$ \\[\sskip]
    \hline
    \raisebox{\sskip}{\phantom{!}} twice the surface area & the degree \\
    $2a=2a(P)$ & $d=d(X_P)$ \\[\sskip]
    the number of interior lattice points & the sectional genus \\
    $i=i(P)$ & $s(X_P)$ \\[\sskip]
    the number of lattice points & the ambient dimension + 1\\
    $n=n(P)$ & $\dim {\mathbb P}^{n-1} + 1$ 
    \\[\sskip]
    Pick's formula & Riemann-Roch \\
    $\displaystyle a = i + \frac{b}{2} -1$ & $\displaystyle d = n+s-2$
    \\[\sskip]
    the number of boundary lattice points &
    the anti-canonical degree of $H$ \\
    $b=b(P)$ & $-HK$
  \end{tabular}
  \medskip
\end{center}






\section{The adjoined pair} \label{sec:adjoined-pair}

We need some concepts on adjunction theory for rational
surfaces. We will use \cite{Reid93} as the basic reference
for adjunction theory for surfaces.

We consider rational surfaces $F$, possibly singular, in projective
space $\P^N$, $N>0$.
There is a resolution of singularities $f:S\to F\subset\P^N$ with
nonsingular $S$, and the pullback of the line bundle ${\mathcal O}(1)$
defines a nef and big divisor class $H\in\Pic(S)$.
Working with nonsingular surfaces and nef and big divisor classes is
technically easier than working with surfaces with arbitrary
singularities, so we think of $F$ as being represented by the pair
$(S,H)$. Such a pair is called a polarized surface, and $H$ is called
the polarization divisor class. We will always require that the
polarization divisor class is nef, and the adjunction process starts
with a polarization divisor which is nef and big.

As the resolution of singularities is not unique, we may have
non-isomorphic polarized surfaces representing the same projective and
possibly singular surface. There is, however, a minimal one
$(S_0,H_0)$. For any other polarized surface $(S_1,H_1)$ representing
the same singular surface, there is a morphism $g:S_1\to S_0$ such
that $g^\ast H_0=H_1$. Minimality of a polarized surface $(S,D)$ is
characterised by the absence of $-1$-curves $E$ such that $EH=0$.

Adjunction is an iterative process to replace a polarized surface
$(S,H)$, with $H$ nef and big, by another polarized surface
$(\Sa,\Ha)$, which is ``smaller'' in a certain sense. When the process
ends, we have reached a particularily simple situation. More
precisely, adjunction terminates if $H+K$ is not effective. Because 
of the formula $s(H)=h^0(H+K)$ for nef and big divisors on a rational
surface, it follows that either $H$ is not big or that the genus of
$H$ is $0$. 

In the other case, adjunction proceeds in two steps. First, we
transform $(S,H)$ into a minimal pair by successively blowing down all
$-1$-curves orthogonal to the polarization divisor. This produces a
birational morphism $f:S\to \Sa$ such that $H$ is orthogonal to the
kernel of $f_\ast$. Second, we set $\Ha:=f_\ast(H)+\Ka$, where $\Ka$
is the canonical divisor class of $\Sa$.

\begin{lemma}
If adjunction is defined, i.e., if $H+K$ is effective, then $\Ha$ is
again nef.
\end{lemma}

\begin{proof}
This is well-known (see, e.g.~\cite{Reid93}, Proof of Theorem D.3.3);
we give the proof just for the sake of completeness. Assume that $C$
is a prime divisor in $\Sa$ such that $C\Ha<0$. Then $C\Ka<0$. If
$C^2\ge 0$, then Riemann-Roch implies $h^0(C)>1$, hence $C$ moves in a
linear system. Then it cannot have negative intersection with the
effective divisor class $\Ha$. Hence $C^2<0$. By the genus formula,
$C^2+C\Ka\ge -2$. This leaves just room for one case, namely
$C^2=C\Ka=-1$ and $Cf_\ast H=0$. But this contradicts minimality of
the pair $(\Sa,f_\ast H)$.
\end{proof}

The adjunction process is finite, because on any rational surface
there is a nef divisor class $L$ which satisfies $LK<0$, namely the
pullback of the class of lines along the inverse of a rational
parametrization. Then $LH/(-LK)$ is an upper bound for the number of
possible adjunction steps with initial polarized surface $(S,H)$. 

We set
$$ \level(S,H) := \sup \left\{ \frac pq \ : \ 
  qH+pK \text{ effective } \right\} . $$
Then the possible number of adjuntion steps is $\lfloor\level\rfloor$,
the largest integer less than or equal to $\level$.
For the final polarized surface, we have three cases,
by Theorem D.4.1 in \cite{Reid93}. Either
\begin{enumerate}
\item $\level \not\in \mathbb{N}$, but $2\level \in \mathbb{N}$ or
  $3\level \in \mathbb{N}$ and
  $H^{(\lfloor\level\rfloor)}$ is a big divisor class of genus
  $0$, or 
\item $\level \in \mathbb{N}$ and $(H^{(\level)})^2=0$ and either
  \begin{enumerate}
  \item $H^{(\level)}=0$ or
  \item $H^{(\level)}=kP$ for some integer $k>0$ and class $P$ of a
    pencil of genus $0$.
  \end{enumerate}
\end{enumerate}
For all cases, there are toric examples.

\section{The proof}

\subsection{Intersection theory}
We recall two well-known facts on rational surfaces which we will
need in the proof. We write $\rho$ for the rank of the Picard group of the
rational surface under consideration.

\begin{proposition}
$K^2+\rho=10$, and $K^2=9$ $\Rightarrow$ $S = \P^2$.
\end{proposition}

\begin{proof}
In a single blowing up, the number $K^2$ decreases by 1 and the number
$\rho$ increases by 1. Because every birational map is a composition
of blowing ups and their inverses, it follows that $K^2+\rho$ is a
birational invariant. It assumes the value 10 for
$S=\P^2$, hence it is 10 for all rational surfaces.

If $K^2=9$, then the Picard rank must be 1. By the classification of
minimal rational surfaces, $S$ must be $\P^2$.
\end{proof}

We will formulate the proof in terms of the genus $s$ of the
intersection of $S$ with a generic hyperplane in $|H|^*$. The
parameters for $(S,H)$ and $(\Sa,\Ha)$ are related as follows.

\begin{proposition}\label{prop:s=s1+b1}
$\ba+\sa=s$.
\end{proposition}

\begin{proof}
By definition, $\ba=-\Ha\Ka$, and by the genus formula,
$\sa=\frac12\Ha(\Ha+\Ka)+1$.
Let $f:S\to\Sa$ be the minimalisation morphism.  Then 
\[ \ba+\sa=\frac{\Ha(\Ha-\Ka)}{2}+1=\frac{\Ha f_\ast(H)}{2}+1 \]
\[	=\frac{f^\ast(\Ha)H}{2}+1 = \frac{(H+K)H}{2}+1 = b. \]
\end{proof}

\subsection{The induction step}
\begin{lemma}\label{lemma:step}
  Suppose $\Ha$ is big, and denote $\ba$ the anti-canonical degree of
  $\Ha$ on $\Sa$. Then $ b \le  \ba + 9 $ with equality if and only if
  $S = \P^2$.
\end{lemma}
\begin{proof}
If we intersect
$$ H = f^*f_*H = f^*(\Ha-\Ka) $$
with $-K$, we get
\begin{align*}
  b = -HK &= -f^*(\Ha-\Ka)K
  = -(\Ha-\Ka)f_*K \\
  &= -\Ha\Ka + \Ka\Ka = \ba + 10 - \rhoa \,.
\end{align*}
\end{proof}

\begin{theorem}
  Let $H$ be a nef and big divisor on the smooth rational surface
  $S$. Let $\level$ be the level, $s$ the sectional genus of $(S,H)$,
  and let $b = -KH$, $d=H^2$. Then
  \begin{equation} \label{ineq:homog}
    2\level b \le d + 9 \level^2\,,
  \end{equation}
  or equivalently
  \begin{equation} \label{ineq:orig}
    (2\level-1) b \le 2s + 9\level^2 - 2 \,.
  \end{equation}
\end{theorem}

\begin{proof}
  The validity of the statement is preserved if we replace $H$ by $qH$
  for some $q>0$. This is apparent in~\eqref{ineq:homog}.
  So we may assume that the level is integral and proceed by induction
  on $\level$.

  For $\level=1$, the statement is equivalent to $-HK\le H^2+HK+9$.
  This is equivalent to $K^2\le (H+K)^2+9$. But $H+K$ is nef and effective,
  hence $(H+K)^2\ge 0$, so the statement is a consequence of $K^2\le 9$.
  
  If $\level \ge 2$, we have by Lemma~\ref{lemma:step}, induction, and
  Proposition~\ref{prop:s=s1+b1} in that order,
  \begin{align*}
    (2\level-1) b &\le (2\level-1) b^{(1)} + 9(2\level-1) \\
    &= 2 b^{(1)} + (2(\level-1)-1) b^{(1)} + 9(2\level-1) \\
    &\le 2 b^{(1)} + 2s^{(1)} + 9(\level-1)^2 - 2 + 9(2\level-1) \\
    &= 2s + 9\level^2 - 2 \,.
  \end{align*}
\end{proof}

\section{Concluding remarks} \label{sec:conclusion} \noindent

%
%

As Wouter Castryck points out~\cite[(2.7)]{CastryckMovingOut},
it would probably yield much stronger bounds if one found a way to
incorporate the parameter
$$ v := \rho + 2 -\sum_{x \in \Sing F} \mult(x) $$
into the induction inequality of Lemma~\ref{lemma:step}.
Here, the sum runs over the singular points of the image $F$ of $S$ in
$|H|^*$, and the multiplicity $\mult(x)$ of such a point $x$ is the
number of exceptional divisors in its minimal resolution. In the toric
case, $v$ is just the number of vertices of the lattice polygon.


As an application of this circle of ideas, one can estimate the
smallest degree of a parametrization of a rational surface. To this
end, \cite{Schicho98d} bounds the level $\level$ in terms
of the degree $d$ of $S$. The exponent $4$ in this bound is
potentially not optimal. For toric surfaces, there is even a linear 
bound.
The area of a lattice polgon drops at least by~$3$ in each adjunction
step, so the level is bounded by 2/3 times the degree.

\bibliographystyle{plain}
\bibliography{josef}

\begin{thebibliography}{1}

\bibitem{CastryckMovingOut}
W.~Castryck.
\newblock Moving out the edges of a lattice polygon.
\newblock {\em Discr. Comp. Geom.}, 47:496--518, 2012.

\bibitem{HS2i+7}
C.~Haase and J.~Schicho.
\newblock Lattice polygons and the number $2i+7$.
\newblock {\em Math. Monthly}, 2009.

\bibitem{Reid93}
M.~Reid.
\newblock Chapters on algebraic surfaces.
\newblock In J.~Koll{\'a}r, editor, {\em Complex Algebraic Varieties}, pages
  1--154. AMS, 1997.

\bibitem{Schicho98d}
J.~Schicho.
\newblock A degree bound for the parameterization of a rational surface.
\newblock {\em J. Pure Appl. Alg.}, 145:91--105, 1999.

\bibitem{Scott}
P.R. Scott.
\newblock On convex lattice polygons.
\newblock {\em Bull. Austral. Math. Soc.}, 15:395--399, 1976.

\end{thebibliography}
\setlength{\parindent}{0pt}

\end{document}